\documentclass[12pt,titlepage]{article}

\usepackage{graphicx,color}
\usepackage{amsmath,amssymb,amsthm}
\newcommand{\proofka}{{\noindent \it Proof. }}
\newtheorem{thm}{Theorem}
\newtheorem{lemma}[thm]{Lemma}
\newtheorem{cor}[thm]{Corollary}

\newtheorem{prop}[thm]{Proposition}

\newtheorem{defi}{Definition}

\newtheorem{remark}{Remark}
\newtheorem{ques}{Question}

\newtheorem*{cl*}{Claim}

\def\beq{\begin{equation}}\def\eeq{\end{equation}}
\def\beqn{\begin{eqnarray}}\def\eeqn{\end{eqnarray}}

\def\qed{\ifhmode\unskip\nobreak\fi\quad\ifmmode\Box\else$\Box$\fi}

\newcommand{\KG}{{\rm KG}}
\newcommand{\SG}{{\rm SG}}

\title{Alternating odd cycles and orientations of Kneser-like graphs}
\author{\hfill N\'ora Alm\'asi\thanks
{Department of Computer Science and Information Theory,
Faculty of Electrical Engineering and Informatics,
Budapest University of Technology and Economics
{\tt almasi.nora@cs.bme.hu}}
\and G\'abor Simonyi\thanks{HUN-REN Alfr\'ed R\'enyi Institute of Mathematics, Budapest, Hungary and
Department of Computer Science and Information Theory,
Faculty of Electrical Engineering and Informatics,
Budapest University of Technology and Economics. Research partially supported by the National Research, Development and Innovation Office (NKFIH) grant K--132696 of NKFIH Hungary; {\tt simonyi@renyi.hu}
}
}

\date{}

\begin{document}

\maketitle

\begin{abstract}

We call an oriented odd cycle alternating if it has exactly one vertex whose in-degree and out-degree are both positive. 
In this paper, we investigate whether certain graphs admit an orientation that avoids alternating odd cycles as subgraphs, or one in which all their shortest odd cycles become alternating. 
Our focus is on topologically $\chi$-chromatic graphs, that is, graphs for which the topological method yields a sharp lower bound on the chromatic number. 
We present results for several graph families, including Kneser graphs, Schrijver graphs, and generalized Mycielski graphs.


\end{abstract}

\section{Introduction}

An oriented cycle is called alternating if each of its vertices, except at most one, is either a source or a sink.
This means for a cycle of length $\ell$ that $\lfloor\ell/2\rfloor$ of the vertices have outdegree $0$ and also $\lfloor\ell/2\rfloor$ of the vertices have indegree $0$.

Alternating odd cycles play a similar role with respect to the directed local chromatic number (for its definition see Subsection~\ref{subsect:diloc}) as odd cycles do with respect to the chromatic number (or the local chromatic number) of undirected graphs: it depends on their presence, whether the value of the directed local chromatic number is just $2$ or larger. 
We mention that alternating odd cycles also came up in different contexts, see \cite{CPR,GH}.

In this paper, we are interested in the (non-)appearance of alternating odd cycles in orientations of some interesting graphs, in particular, those for which the topological method introduced by Lovász in \cite{LLKn} gives a sharp lower bound for the chromatic number.

It is easy to show (and is already done in \cite{STdir}) that $3$-colorable graphs can be oriented so that no alternating odd cycles appear. 
Using some known results proven with the help of topological tools about graphs mentioned in the previous paragraph, one finds that if such a graph has chromatic number at least $5$, then alternating odd cycles are unavoidable in its orientations. 
(This result is, in fact, already implicit in \cite{STdir}.) 
So our main focus will be on $4$-chromatic graphs. 
We will show that every $4$-chromatic Schrijver graph (see Subsection~\ref{subsect:toptchrom} for the definition) can be oriented so that it contains no alternating odd cycles (except $K_4$, which is also a Kneser graph). 
We achieve this by proving an equivalent statement: that the directed local chromatic number of the resulting oriented graph is $2$. 
Moreover, we show we can achieve this orientation using only $4$ colors. 
We establish similar results for Mycielski graphs as well.
Conversely, we also show that $4$-chromatic Kneser graphs have no such orientation that admits a $4$-coloring attaining directed local chromatic number $2$. 
We will see that this is equivalent to saying that $4$-chromatic Kneser graphs do not admit a homomorphism into the so-called symmetric shift graph $S_4$ (for its definition, see Section~\ref{sect:avoid}).
In contrast, all $4$-chromatic Schrijver graphs that are not also Kneser graphs (that is, all except $K_4$) do admit such a homomorphism.

In Section~\ref{sect:altall} we investigate the question of when a graph can be oriented so that all the shortest odd cycles in it become alternating. 
We will show, for example, that this is always possible for certain Kneser graphs.  


\section{Avoiding alternating odd cycles}\label{sect:avoid}

\subsection{Directed local chromatic number}\label{subsect:diloc}

The local chromatic number $\psi(G)$ of a graph $G$ was introduced in \cite{EFHKRS}. 
It can be defined as follows. For a proper coloring $c$ and a subset $U$ of the vertex set $V(G)$ of graph $G$, let $c(U):=\{c(u): u\in U\}$ denote the set of colors appearing on vertices belonging to $U$.
The (open) neighborhood of a vertex $v$ in a graph $G$ is the set of vertices adjacent to $v$, while the closed neighborhood includes $v$ itself along with its neighbors.
We denote the closed neighborhood of a vertex $v$ in $G$ by $N_G[v]:=\{v\}\cup \{u\in V(G): \{v,u\}\in E(G)\}$, or when the graph $G$ is understood from the context, we may write simply $N[v]$.

\begin{defi}\label{defi:psi}{\rm (\cite{EFHKRS})}
The local chromatic number of a graph $G$ is defined as
$$\psi(G)=\min_c\max\{|c(N[v])|: v\in V(G)\},$$
where $c$ runs over all proper colorings of $G$.
\end{defi}

While it is obvious from the definition that $\psi(G)$ cannot exceed the chromatic number $\chi(G)$ of the graph, somewhat surprisingly, it can be just $3$ for graphs of arbitrarily large chromatic number.

\begin{thm}\label{thm:locgap} {\rm (\cite{EFHKRS})}
  For every positive integer $m\ge 3$ there exists graph $G$ with $\psi(G)=3$ and $\chi(G)\ge m$.
\end{thm}

Note that we cannot write $2$ in place of $3$ in the above theorem: it is easy to see that $\psi(G)=2$ implies that $G$ is bipartite.

For a digraph $D$ let $N_+[v]$ denote the closed outneighborhood of vertex $v$, i.e., $N_+[v]:=\{u\in V(G): (v,u)\in E(D)\}\cup\{v\}$. (Note that $(v,u)$ refers to an ordered pair of vertices, that is, an edge oriented from $v$ towards $u$.)
It is straightforward to generalize Definition~\ref{defi:psi} to directed graphs as it is done in \cite{KPS} by considering outneighborhoods in place of neighborhoods.

\begin{defi}\label{defi:dpsi}{\rm (\cite{KPS})}
  The directed local chromatic number of a digraph $D$ is defined as
  $$\psi_d(D)=\min_c\max\{|c(N_+[v])|: v\in V(D)\},$$
  where $c$ runs over all proper colorings of (the underlying undirected graph of) $D$.
\end{defi}

The analog of Theorem~\ref{thm:locgap} is the following.

\begin{thm}\label{thm:dilocgap}
  For every positive integer $m\ge 2$ there exists an oriented graph $\overrightarrow{D}$ with $\psi_d(\overrightarrow{D})=2$ and $\chi(D)\ge m$, where $D$ is the underlying undirected graph of $\overrightarrow{D}$.
\end{thm}

The proof of Theorem~\ref{thm:dilocgap} is analogous to that of Theorem~\ref{thm:locgap} and also easily follows from Theorem 7 and Proposition 8 in \cite{STdir}.

\begin{defi}
  The symmetric shift graph $S_m$ is defined on the vertex set
  $$V(S_m)=\{(i,j): 1\le i,j\le m, i\neq j\}$$
  having the edge set
  $$E(S_m)=\{\{(i,j)(k,\ell)\}: j=k\ {\rm or}\ i=\ell\}.$$
\end{defi}

The classical shift graph was originally defined by Erdős and Hajnal \cite{EH}. 
It is the subgraph of the graphs $S_m$ induced by the vertices $(i,j)$ satisfying $i<j$.

A homomorphism from a (di)graph $F$ to a (di)graph $G$ is an edge-preserving map from $V(F)$ to $V(G)$. 
We will denote the existence of such a homomorphism by $F\rightarrow G$.
When no such homomorphism exists, we write $F\nrightarrow G$.

The graphs $S_m$ have a natural directed version $\overrightarrow{S}_m$ in which edge $\{(i,j),(j,k)\}$ for $i\neq k$ is oriented from $(i,j)$ to $(j,k)$ while edges $\{(i,j),(j,i)\}$ are substituted by a pair of oppositely directed edges. 
It is noted already in \cite{STdir} (cf. the paragraph after the proof of Proposition 8) that a digraph $F$ admits a coloring with $m$ colors realizing $\psi_d(F)\le 2$ if and only if it admits a homomorphism into $\overrightarrow{S}_m$ and thus $\overrightarrow{S}_m$ is simply the homomorphism universal graph denoted $U_d(m,2)$ in \cite{KPS}. 
From these facts it is just a little step forward to see the following statement that we give for further reference.

\begin{prop}\label{prop:shiftuni}
An undirected graph $G$ admits an orientation $\overrightarrow{G}$ satisfying $\psi_d(\overrightarrow{G})\le 2$ that can be attained using $m$ colors if and only if $G\rightarrow S_m$.
\end{prop}

\proofka
It follows from the foregoing that if $G$ admits such an orientation $\overrightarrow{G}$, then $\overrightarrow{G}\rightarrow\overrightarrow{S_m}$, and then clearly the same mapping between the vertex sets witnesses $G\rightarrow S_m$.

On the other hand, if $G\rightarrow S_m$ holds then considering the same map $\phi: V(G) \to V(S_m)$ as a map to $V(\overrightarrow{S_m})$ and ``pulling back'' the orientation of the edges of $\overrightarrow{S_m}$
to their preimages (i.e., orienting $\{a,b\}\in E(G)$ as $(a,b)$ if the edge $\{\phi(a),\phi(b)\}$ is oriented as $(\phi(a),\phi(b))$, or arbitrarily if both $(\phi(a),\phi(b))$ and $(\phi(b),\phi(a))$ are edges of $\overrightarrow{S_m}$) we obtain an oriented version $\overrightarrow{G}$ of $G$ that admits a homomorphism to $\overrightarrow{S_m}$ and thus has $\psi_d(\overrightarrow{G})\le 2$ attainable with $m$ colors.
\hfill$\Box$

As we already mentioned in the introduction, a digraph $\overrightarrow{G}$ has $\psi_d(\overrightarrow{G})\le 2$ if and only if it does not contain an alternating odd cycle. This is proven in \cite{STdir} (see Proposition 14 and its proof in \cite{STdir}). This together with Proposition~\ref{prop:shiftuni} means that we can orient a graph avoiding alternating odd cycles if and only if it admits a homomorphism into the shift graph $S_m$ for some $m$.

Following \cite{STdir}, for an undirected graph $G$ we will use the notation $\psi_{d,{\rm min}}(G)$ meaning the minimum possible value of $\psi_d(\overrightarrow{G})$ over all orientations of $G$.
\medskip

The following observation is immediate and already appears in \cite{STdir} (in a more general form). Still, we give the proof for completeness.

\begin{prop}\label{prop:3chrom} {\rm (\cite{STdir})}
If $G$ is a graph containing at least one edge and having chromatic number at most $3$, then $\psi_{d,{\rm min}}(G)=2.$ In particular, it admits an orientation avoiding alternating odd cycles.
\end{prop}

\proofka
Consider a proper coloring of $G$ by at most $3$ colors, say red, blue and green. 
Every edge connects two differently colored vertices. 
If those two colors are red and blue, orient the edge towards the blue one, for blue and green towards the green one, and for red and green towards the red one. 
Then, keeping the same coloring, every vertex with outdegree at least $1$ has exactly $2$ colors in its closed outneighborhood. 
This implies that there is no alternating odd cycle and, since the graph is not edgeless, that $\psi_{d,{\rm min}}(G)=2$.
\hfill$\Box$

Note that the proof above is equivalent to taking a homomorphism from $G$ to $K_3$ that exists by $3$-colorability, then considering the target graph $K_3$ a cyclically oriented triangle and pulling back the orientation of its edges to their preimages.

\subsection{Topological $t$-chromaticity}\label{subsect:toptchrom}

The topological method giving lower bounds on the chromatic number of graphs was introduced by Lovász in his pioneering paper \cite{LLKn}. 
In that work, he proved Kneser's conjecture concerning the chromatic number of a family of graphs that later came to be known as Kneser graphs.

\begin{defi}\label{defi:KG}
The Kneser graph $\KG(n,k)$ with parameters $n\ge 2k$ is defined on the vertex set ${[n]\choose k}$ consisting of all $k$-element subsets of the $n$-element set $[n]=\{1,\dots,n\}$ with edge set
$$E(\KG(n,k))=\{\{A,B\}: A,B\in {[n]\choose k}, A\cap B=\emptyset\}.$$
\end{defi}

Lovász proved that $\chi(\KG(n,k))=n-2k+2$, confirming a conjecture of Kneser, who had earlier established this value as an upper bound \cite{Kne}.
Lovász's proof used the Borsuk-Ulam theorem from algebraic topology. 
The method he introduced has different implementations since then, for more about this we refer to the paper \cite{MZ}, the recent nice survey \cite{DM}, and to Matou\v{s}ek's excellent book \cite{Matbook}. 
In \cite{ST1}, graphs for which one particular implementation of the topological method gave a lower bound $t$ for the chromatic number were called topologically $t$-chromatic.
For the understanding of this paper we do not need the precise definition of the topological notions involved, it is enough to know the result called Zig-zag theorem in \cite{ST1} (cf. also \cite{KF} where it is already proven for Kneser graphs) which states that in any proper coloring of a topologically $t$-chromatic graph a completely multicolored complete bipartite graph $K_{\lfloor{t/2}\rfloor,\lceil{t/2}\rceil}$ must occur irrespective of the number of colors used. (The Zig-zag theorem actually states more, giving some restrictions on which colors should be on the two sides of this bipartite graph once the colors are ordered, but this is not needed for us at this point.) This easily implies the following result proven in \cite{STdir}.

\begin{thm}\label{thm:tnegyed} {\rm (see as Theorem 20 in \cite{STdir})}
If $G$ is a topologically $t$-chromatic graph with $t\ge 2$, then
$$\psi_{d,{\rm min}}\ge \lceil{t/4}\rceil +1.$$
\end{thm}

In \cite{STcolful}, some main examples of topologically $t$-chromatic graphs with chromatic number equal to $t$ are presented, including the following three. 
The first has already been defined; the remaining two will be defined after listing all three.

\begin{itemize}
\item Kneser graphs $\KG(n,k)$ for $t=n-2k+2$;
\item Schrijver graphs $\SG(n,k)$ for $t=n-2k+2$;
\item Mycielski and generalized Mycielski graphs $M_{r_1,\dots,r_k}(K_2)$ for $t=k+2$;
\end{itemize}

Schrijver graphs $\SG(n,k)$ were defined by Schrijver \cite{Schr} as induced subgraphs of Kneser graphs $\KG(n,k)$ of the same parameters and he proved that $\SG(n,k)$ is a vertex-color-critical subgraph within $\KG(n,k)$ sharing its chromatic number.

\begin{defi}\label{defi:Sch}
The Schrijver graph $\SG(n,k)$ with parameters $n\ge 2k$ is the induced subgraph of $\KG(n,k)$ on the vertex set consisting of the so-called stable $k$-subsets:
$$V(\SG(n,k)= \left\{ A\in {[n]\choose k} \middle| \ \forall\ i\in [n-1]:\  \{i,i+1\}\nsubseteq A, \{1,n\}\nsubseteq A\right\}.$$
\end{defi}
\medskip

\begin{defi}\label{defi:genMyc}
The $r$-level generalized Mycielskian $M_r(G)$ of a graph $G$ is defined on the vertex set
$$V(M_r(G))=\{(v,i): v\in V(G), 0\le i\le r-1\}\cup\{z\}$$ with edges
$$E(M_r(G))=\{\{(u,i),(v,j)\}: \{u,v\}\in E(G)\ {\rm and}\ (|i-j|=1\ {\rm or}\ i=j=0)\}\cup$$
$$\{\{z,(v,r-1)\}: v\in V(G)\}.$$
\end{defi}

When $r=2$ we simply call $M(G)=M_2(G)$ the Mycielskian of $G$. $M_{r_1,\dots,r_k}(G)$ stands for the iterated Mycielskian $M_{r_k}(M_{r_{k-1}}(\dots M_{r_1}(G)))$. Following Stiebitz \cite{Stieb} it is shown in \cite{GyJS, Matbook} that if $G$ is topologically $t$-chromatic, then $M_r(G)$ is topologically $(t+1)$-chromatic for any positive integer $r$ and so
$M_{r_1,\dots,r_k}(G)$ is topologically $(t+k)$-chromatic for every sequence of positive integers $r_1,\dots,r_k$. By the term Mycielski graphs and generalized Mycielski graphs, we refer to the iterated Mycielskians and generalized Mycielskians of $G=K_2$. Since $K_2$ is topologically $2$-chromatic, generalized Mycielskians are topologically $t$-chromatic if $t-2$ iterations of the construction are performed. It is easy to see that $\chi(M_r(G))\le\chi(G)+1$, so the graphs mentioned in the previous sentence also have chromatic number $t$. Note that Definition~\ref{defi:genMyc} also makes sense for $r=1$, $M_1(G)$ is simply the graph obtained from $G$ by adding a new vertex adjacent to all vertices of $G$.
\medskip

We have already seen in Theorem~\ref{thm:tnegyed}, that for $t\ge 5$ all topologically $t$-chromatic graphs have $\psi_{d,{\rm min}}\ge 3$. Combining this with Proposition~\ref{prop:3chrom}, we can state the following Corollary.

\begin{cor}\label{cor:csaka4}
Let $G$ be a topologically $t$-chromatic graph with $\chi(G)=t$.
If $t\le 3$ then $G$ admits an orientation avoiding alternating odd cycles, while for $t\ge 5$ every orientation of such a $G$ contains an alternating odd cycle.
\end{cor}

The only case not covered by Corollary~\ref{cor:csaka4} is when $t=4$. This is what we concentrate on in the next subsection.



\subsection{The exceptional case: Topologically $4$-chromatic graphs}

We first recall that \cite{STdir} presents a general result implying that certain 4-chromatic Schrijver graphs and generalized Mycielski graphs admit orientations with directed local chromatic number $2$; in other words, orientations that contain no alternating odd cycles.

\begin{defi}\label{defi:swide}
A vertex-coloring of a graph $G$ is called $s$-wide if the so-colored $G$ contains no walk of length $2s-1$ between two (not necessarily distinct) vertices of the same color.
\end{defi}

Note that a coloring being $1$-wide simply means that it is proper.
The $2$-wide $t$-colorability of $t$-chromatic graphs was first investigated in \cite{GyJS}. 
The $3$-wide colorings were shown to be relevant concerning the local chromatic number in \cite{ST1}, but $s$-wide coloring is also considered there in general. 
The concept was also independently introduced in the unpublished paper by Baum and Stiebitz \cite{BS}, see also Chapter 8 of the recent book by Stiebitz, Schweser and Toft \cite{SST} where wide colorings are called odd colorings.

Corollary 23 of \cite{STdir} states that if a non-edgeless graph admits a $2$-wide $4$-coloring, then it has $\psi_{d,{\rm min}}=2$, and the proof shows that this can be realized with the $2$-wide $4$-coloring itself. 
Results from \cite{ST1} imply that $4$-chromatic Schrijver graphs $\SG(2k+2,k)$ admit such $2$-wide $4$-colorings if $k$ is large enough. 
Here, however, we prove that every $4$-chromatic Schrijver graph except $K_4$ (which is also a Kneser graph) has $\psi_{d,{\rm min}}=2$, even those that are known not to admit a $2$-wide $4$-coloring.
Moreover, we show that this value can always be attained using $4$ colors.
In contrast, we also show that no orientation of $\KG(n,k)$ allows a proper $4$-coloring in which every vertex has a monochromatic outneighborhood. 
In line with Proposition~\ref{prop:shiftuni} we state this result in the following form.

\begin{thm}\label{thm:toS4}
  \par\noindent
  1. For every $k\ge 2$ we have $$\SG(2k+2,k)\rightarrow S_4.$$
  \par\noindent
  2. It holds for any positive integer $k$ that $$\KG(2k+2,k)\nrightarrow S_4.$$
\end{thm}

In order to prove Theorem \ref{thm:toS4}, we will use the following result of Peng-An Chen \cite{PAC} (see also \cite{CLZ, LZ} for simplified proofs).

\begin{thm}\label{thm:altKncol} {\rm (Alternative Kneser coloring theorem \cite{PAC})}
  Let $n\ge 2k\ge 2$ and assume that $c$ is a proper coloring of the graph $\KG(n,k)$ with $\chi(\KG(n,k))=n-2k+2$ colors. 
  Then there exists two disjoint $(k-1)$-element subsets $S$ and $T$ of $[n]$ such that for all $i\in[n]\setminus(S\cup T)$ we have $c(S\cup\{i\})=c(T\cup\{i\})$. 
  Note that for $i\ne i'$ these colors must be distinct.
\end{thm}

\begin{figure} [h!]
    \centering
    \includegraphics[width = 6 cm]{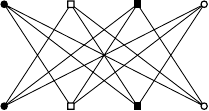}
    \caption{By Theorem \ref{thm:altKncol}, any properly $4$-colored $\KG(2k+2,k)$ contains this graph $G$ as a subgraph.}
    \label{fig:k44}
\end{figure}

\medskip
\par\noindent
{\it Proof of Theorem~\ref{thm:toS4}.}
Let $G$ denote the bipartite graph formed by removing a perfect matching from $K_{4,4}$.
Fix a $4$-coloring of $G$, where both of its partite classes contain a vertex of each color, just as shown in Figure \ref{fig:k44}.
To prove the second statement, note that Chen's Alternative Kneser coloring theorem (Theorem~\ref{thm:altKncol}) implies that any properly $4$-colored $\KG(2k+2,k)$ contains $G$ as a subgraph, with both of its partite classes containing a vertex of each color.
In particular, this means that if there were an orientation and a $4$-coloring of $\KG(2k+2,k)$ realizing $\psi_d=2$, then within this copy of $G$ every vertex could have at most one outgoing edge to another vertex in the same subgraph. 
As a result, the number of edges within this subgraph would be at most $8$. 
However, since $G$ has $12$ edges, this leads to a contradiction.
Therefore, such an orientation and $4$-coloring cannot exist.
In particular, by Proposition~\ref{prop:shiftuni}, $\KG(2k+2,k)$ admits no homomorphism into $S_4$.

It is well-known that for every $k$ one has $\KG(n,k)\to\KG(n-2,k-1)$, see Proposition 6.26 in \cite{HN}, and the analogous statement for Schrijver graphs is even more straightforward. 
Indeed, deleting the largest element of every stable $k$-subset of $[n]$ results in a stable $(k-1)$-subset of $[n-2]$ and since removing elements cannot violate disjointness, this proves $\SG(n,k)\rightarrow \SG(n-2,k-1).$ 
Therefore, the first statement will follow if we prove it for $k=2$, namely, if we show the existence of the homomorphism $\SG(6,2) \rightarrow S_4$. 
This is by Proposition \ref{prop:shiftuni} equivalent to showing that $\SG(6,2)$ has an orientation and a $4$-coloring attaining $\psi_{d,min}(\SG(6,2))=2$.
The appropriate coloring and orientation are illustrated in Figure \ref{fig:sg_6_2}.
\hfill$\Box$

\begin{figure} [h!]
    \centering
    \includegraphics[width = 8.5 cm]{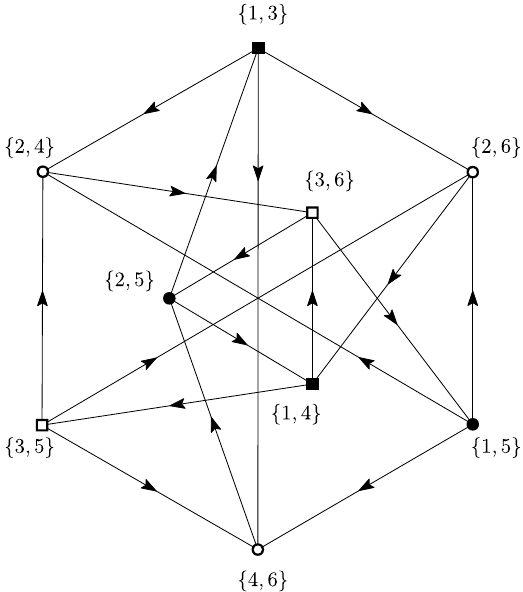}
    \caption{Orientation and a $4$-coloring attaining $\psi_{d,min}(\SG(6,2))=2$.}
    \label{fig:sg_6_2}
\end{figure}

Note that for any graph $H$ and any topologically $t$-chromatic graph $G$, we have that if a homomorphism $G \rightarrow H$ exists, then $H$ is also topologically $t$-chromatic.
Also note that for all $m \geq 4$ we have $S_4 \subseteq S_m$, thus a homomorphism $S_4 \rightarrow S_m$ exists.
Therefore the first statement of Theorem~\ref{thm:toS4} also implies that $S_4$ (and all $S_m$ with $m\ge 4$) is topologically $4$-chromatic contrary to an unproven remark in \cite{STdir}, where it is shown, on the other hand, that $S_m$ is not topologically $5$-chromatic.

\begin{remark}
{\rm Since the structure of $4$-chromatic Schrijver graphs is well-understood, cf. \cite{BB1,BB2,STActa}, one can also prove the first statement of Theorem~\ref{thm:toS4} more directly, that is without referring to $\SG(n,k)\rightarrow\SG(n-2,k-1).$
For any $k\ge 2$, the graph $\SG(2k+2,k)$ contains a unique (induced) complete bipartite graph $K_{k+1,k+1}$ (between vertices containing $k$ odd elements and vertices containing $k$ even elements of $[2k+2]$). 
Call the two partite classes of this bipartite graph $A$ and $B$. 
It is also true that all vertices in $A\cup B$ are incident to exactly one edge whose other endpoint is not in $A\cup B$.
Moreover, this vertex is different for all vertices in $A$ and also different for all vertices in $B$. 
(In fact, when $k>2$, all $2k+2$ vertices are distinct, though this is not essential for our argument.) 
Consider the subgraph induced on $V(\SG(2k+2,k))\setminus (A\cup B)$. 
It is $3$-colorable, because Schrijver graphs are vertex-critical with respect to the chromatic number. 
Color this subgraph with at most $3$ colors $c_0, c_1, c_2$ and orient the edges between them so that edges connecting vertices colored $c_i$ and $c_{i+1}$ are oriented toward the latter (where addition is intended modulo $3$).
Color the vertices in $B$ by $c_3$ and orient all edges between $A$ and $B$ towards the vertex in $B$. 
Let $U:=V(\SG(2k+2,k))\setminus (A\cup B)$.
We orient the edges connecting vertices in $B$ to a vertex in $U$ outwards from $B$, and edges connecting vertices in $U$ to a vertex in $A$ are oriented towards $A$. 
If the so-oriented in-neighbor of a vertex $a\in A$ is colored $c_i$, we color $a$ with color $c_{i+1}$.
This results in a properly colored oriented graph where every $c_i$-colored vertex of $U$ has only $c_{i+1}$-colored vertices in its out-neighborhood, vertices in $A$ have only $c_3$-colored vertices in their outneighborhood and vertices in $B$ have outdegree $1$, so their outneighborhood is also monochromatic. 
This proves $\psi_{d,{\rm min}}(\SG(2k+2,k))=2$ for all $k\ge 2$.}
$\Diamond$
\end{remark}

Note that Theorem~\ref{thm:toS4} does not imply that $\psi_{d,{\rm min}}(\KG(2k+2,k))>2$, only that the value $2$  cannot be attained using $4$ colors. 
At least for the smallest such graph different from $K_4$, which is $\KG(6,2)$, we prove below that it cannot have directed local chromatic number $2$ for any orientation.

  \begin{prop}~\label{prop:KG62}
$$\psi_{d,{\rm min}}(\KG(6,2))=3.$$
  \end{prop}

  \proofka
  Assume for contradiction that $\KG(6,2)$ admits an orientation $\overrightarrow{\KG}(6,2)$ attaining $\psi_d(\overrightarrow{\KG}(6,2))=2$. Note that the neighborhood of any vertex $v$ of $\KG(6,2)$ (to which it is worth thinking about as the complement of the line graph of $K_6$) induces $3K_2$. 
  Thus, in any proper coloring, at most three of them can have the same color, so to attain $\psi_d(\overrightarrow{\KG}(6,2))=2$, the outdegree of any vertex must be at most $3$. 
  The sum of outdegrees is always the total number of edges, which is, in our case, exactly three times the number of vertices. 
  So each outdegree must be exactly $3$. 
  Since each vertex has an indegree of at least one, this also implies that every color class should contain at least $3$ vertices. 
  Since there are $15$ vertices in $\KG(6,2)$, this implies that we can have at most $5$ color classes. 
  We already know from Theorem ~\ref{thm:toS4} that we cannot have exactly $4$ color classes, so the only option is to use $5$ colors, and each color is used exactly three times. Call a color class a {\em star color class} if the three element sets belonging to it have one common element (i.e., as edges of $K_6$ they form a $3$-edge star) and call it a {\em triangle color class} if the pairwise intersections of the three sets in it are all different (i.e., as edges of $K_6$ they form a triangle). For a star color class, there is exactly one vertex whose neighborhood contains all its vertices (there is exactly one edge of $K_6$ that is disjoint from a $3$-edge star), and for a triangle color class, there are exactly three such vertices in $\KG(6,2)$. Since we have $15$ vertices and $5$ color classes, this implies that to have a monochromatic neighborhood for all vertices, we need that all five color classes are triangles. But it is easy to check that the edge set of $K_6$ cannot be partitioned into $5$ triangles. (For example, it would need all vertices in $K_6$ to have an even degree.) So we have
  $\psi_{d,{\rm min}}(\KG(6,2))\ge 3.$ Equality follows by simply observing that every $4$-chromatic graph $G$ has $\psi_{d,{\rm min}}(G)\le 3$. (This is easy to see, but cf. also Proposition 4 in \cite{STdir}.)
  \hfill$\Box$

The following is an immediate consequence of Propositions~\ref{prop:shiftuni} and \ref{prop:KG62}.

\begin{cor}\label{cor:62nohm}
For all $m$ we have
$$\KG(6,2)\nrightarrow S_m.$$
\end{cor}

The corollary generalizes Theorem \ref{thm:toS4} for the special case $\KG(6,2)$, showing that no homomorphism to any symmetric shift graph exists. 
This naturally leads to the question of whether similar results hold for larger 4-chromatic Kneser graphs. 
The situation remains unresolved even for the next smallest example.

\begin{ques}
\label{ques:kg-no}
    Does there exist an integer $m$ such that $\KG(8,3) \rightarrow S_m$? More generally, for which integers $k \geq 3$ does there exist an $m$ such that $\KG(2k+2, k) \rightarrow S_m$?
\end{ques}

\bigskip

Next, we prove that the Grötzsch graph, that is, the $4$-chromatic Mycielski graph $M(C_5)=M(M(K_2))$, just like $4$-chromatic Schrijver graphs, admits a homomorphism to $S_4$. This means that it admits an orientation and a $4$-coloring in which the outneighborhood of every vertex is monochromatic. We will also show that no generalized Mycielskian of the triangle has this property.
The consequences of these observations are summarized in the following theorem. (Note that $M_1(K_2)=K_3$.)

\begin{thm}\label{thm:MyctoS4}
Let $r_1,r_2$ be positive integers. We have
$$\psi_{d,{\rm min}}(M_{r_1,r_2}(K_2))=2$$
if and only if both $r_1, r_2\ge 2$.
In this case directed local chromatic number $2$ can be attained with $4$ colors, that is, $$M_{r_1,r_2}(K_2)\rightarrow S_4,$$
while for every positive integers $r$ and $m$
$$M_{1,r}(K_2)\nrightarrow S_m\ \ {\rm and}\ \ M_{r,1}(K_2)\nrightarrow S_m$$ holds.
\end{thm}

\begin{lemma}\label{lem:GrtoS4}
$$M(C_5)\rightarrow S_4.$$
\end{lemma}

\proofka
We give an orientation and a $4$-coloring of $M(C_5)$, which is the Grötzsch graph, so that every out-neighborhood is monochromatic. This is shown in Figure \ref{fig:Gr}. Note that three of the edges are not oriented on the figure because, whatever way they are oriented, the out-neighborhoods remain monochromatic, so their orientation can be chosen freely. It is easy to check that indeed every out-neighborhood contains only one color.
By Proposition~\ref{prop:shiftuni}, the existence of this orientation and $4$-coloring implies the statement.
\hfill$\Box$

\begin{figure} [h!]
    \centering
    \includegraphics[width = 8 cm]{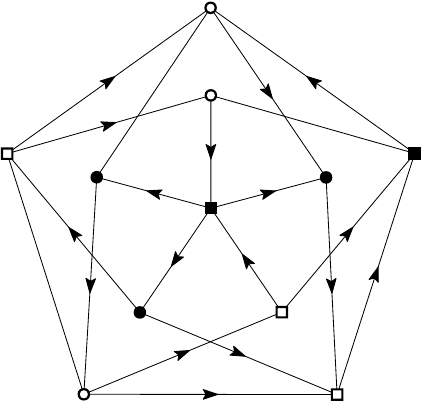}
    \caption{An orientation and $4$-coloring of the Grötzsch graph attaining directed local chromatic number $2$. For three edges, no orientation is shown on the picture because they can be oriented either way.}
    \label{fig:Gr}
\end{figure}

\medskip
\par\noindent
{\em Proof of Theorem~\ref{thm:MyctoS4}}.
In the light of Lemma~\ref{lem:GrtoS4} the first statement follows by showing that for $r_1,r_2\ge 2$ we have $M_{r_1,r_2}(K_2)\rightarrow M_{2,2}(K_2)$. This is immediate if we observe the following two facts. One is that for any graph $G$ and $r'\ge r$ we always have $M_{r'}(G)\rightarrow M_r(G)$.
This follows simply by seeing that the map $\varphi$ that maps the $z$ vertex of $M_{r'}(G)$ to the $z$ vertex of $M_r(G)$, while for the rest of the vertices of $M_{r'}(G)$ it is defined by
$$\varphi(v,i)=(v,\min\{i-(r'-r),0\}).$$
(In the above formula, $\varphi(v,i)$ stands for $\varphi((v,i))$, and a similar convention for the notation of functions taking values on vertices of a generalized Mycielskian will be followed also in later parts of this proof.)
The other fact needed is that if $G\rightarrow H$ then $M_r(G)\rightarrow M_r(H)$ holds for any positive integer $r$.
Indeed, if $f: V(G)\to V(H)$ is a homomorphism from $G$ to $H$, then
$$z_{M_r(G)}\mapsto z_{M_r(H)}, \ \ (v,i)\mapsto (f(v),i)$$
is a homomorphism from $M_r(G)$ to $M_r(H)$.
These two facts and Lemma~\ref{lem:GrtoS4} then imply $$M_{r_1,r_2}(K_2)=M_{r_2}(M_{r_1}(K_2))\rightarrow M_2(M_{r_1}(K_2))\rightarrow M_2(M_2(K_2))=M_{2,2}(K_2)\rightarrow S_4.$$

For the first part of the second statement, note that $M_{1,r}$ is just the wheel $W_{2r+1}$, the graph that consists of a cycle $C_{2r+1}$ plus a vertex $z$ adjacent to all the $2r+1$ vertices of the cycle. Thus $z$ is a vertex of $2r+1$ triangles, all of which should be oriented cyclically to avoid a transitively oriented triangle, which is a special alternating odd cycle. Thus, edges from $z$ to neighboring vertices of the cycle should be oriented in the reverse direction. By the odd length of the cycle, this is impossible.

For the second part of the second statement, consider $M_{r,1}(K_2)=M_r(K_3)$, where $(a,0), (b,0), (c,0)$ correspond to the vertices of the original $K_3$.
Assume for contradiction that an orientation and a coloring $f$ of $M_r(K_3)$ exists that attains directed local chromatic number $2$. Triangles in such an orientation must be cyclic, so we may assume w.l.o.g. that edge $\{(a,0),(b,0)\}$ is oriented towards $(b,0)$, $\{(b,0),(c,0)\}$ towards $(c,0)$, and $\{(c,0),(a,0)\}$ towards $(a,0)$. To keep the triangles induced by vertex subsets $\{(a,0),(b,0),(c,1)\}$, $\{(a,0),(b,1),(c,0)\}$ and $\{(a,1),(b,0),(c,0)\}$ cyclic, the vertex $(a,1)$ should behave towards $(b,0)$ and $(c,0)$ as $(a,0)$, that is edge $\{(a,1),(b,0)\}$ is oriented towards $(b,0)$, edge $\{(c,0),(a,1)\}$ towards $(a,1)$ and the analogous statement is true for $(b,1)$ and $(c,1)$. To keep  outneighborhoods monochromatic this also implies $f(x,1)=f(x,0)$ for all $x\in \{a,b,c\}$. (Clearly $f(a,0), f(b,0), f(c,0)$ are all distinct.)
\medskip\par\noindent
{\em Claim.} We must have $f(x,i)=f(x,0)$ for all $i\in\{0,1,\dots,r-1\}$ and $x\in\{a,b,c\}$.
Furthermore, the edge $\{(x,i),(y,i-1)\}$ should be oriented the same way as $\{(x,0),(y,0)\}$ for all $i\in\{1,2\dots,r-1\}$ and $x,y\in\{a,b,c\}$.
\smallskip\par\noindent
Indeed, we have already seen that the statement is true for $i=1$, so we may assume inductively that it holds up to $i=j-1\ge 1$ and show it for $i=j$. Vertex $(x,j)$ is connected to $(y,j-1)$ and $(w,j-1)$ where $\{y,w\}=\{a,b,c\}\setminus\{x\}$. By the induction hypothesis we have $f(y,j-1)=f(y,0)\neq f(w,0)=f(w,j-1)$. So at most one of the two edges $\{(x,j),(y,j-1)\}$ and $\{(x,j),(w,j-1)\}$ can be oriented away from $(x,j)$. Thus, say, $\{(x,j),(w,j-1)\}$ is oriented towards $(x,j)$, which implies that the color $f(x,j)$ should be the same as the color already appearing in the out-neighborhood of $(w,j-1)$, which is one of the two distinct colors $f(x,j-2)$ and $f(y,j-2)$. But $f(y,j-2)=f((y,j-1)$ by the induction hypothesis, so this color cannot be given to vertex $(x,j)$ as it is adjacent to $(y,j-1)$. Thus $w$ should be that element of $\{a,b,c\}\setminus\{x\}$ for which we have $(x,j-2)$ in the out-neighborhood of $(w,j-1)$ and we must have $f(x,j)=f(x,j-2)=f(x,0)$. The argument also implies that the edge $\{(x,j),(y,j-1)\}$ should be oriented away from $(x,j)$, and thus the Claim is proved.

\smallskip\par\noindent
Applying the Claim for $i=r-1$ we obtain that the three neighbors of vertex $z$, namely $(a,r-1), (b,r-1), (c,r-1)$, have the three distinct colors $f(a,0), f(b,0), f(c,0)$. 
Moreover, each of these vertices already has exactly one of these three colors in its out-neighborhood. 
This implies that $f(z)$ must be assigned a fourth color, meaning that none of the three incident edges can be oriented toward $z$. 
However, this would result in $z$ having all three distinct colors in its out-neighborhood, contradicting our assumption.
\hfill$\Box$

Previously, we showed that the Grötzsch graph admits an orientation without any alternating odd cycle.
This naturally led us to ask whether the same property holds for another well-known graph, the Clebsch graph, which in fact contains many copies of the Grötzsch graph.
The Clebsch graph is also known as the Greenwood-Gleason graph \cite{GG}, in connection with the determination of the Ramsey number $R(3,3,3)=17$.

\begin{prop} \label{prop:cl-no}
    Every orientation of the Clebsch graph admits an alternating odd cycle. 
\end{prop}

Note that this implies that the Clebsch graph does not admit a homomorphism to any symmetric shift graph $S_m$.
The proof of Proposition \ref{prop:cl-no} is not included here, since it relies primarily on an extensive case analysis.

\bigskip
\par\noindent
Another well-known class of topologically $t$-chromatic graphs, whose members have chromatic number $t$ and are also included in the already mentioned list of such graphs in \cite{STcolful}, is the family of rational (or circular) complete graphs $K_{p/q}$ where $t=\left\lceil{p/q}\right\rceil$ odd.
While this family does not include any $4$-chromatic graphs (since $4$ is even), it is interesting to see that the $4$-chromatic rational complete graphs (that is, those with $\left\lceil{p/q}\right\rceil=4$) admit no orientation without alternating odd cycles.

\begin{defi}\label{defi:Kpq}
The rational complete graph $K_{p/q}$ is defined for $p\ge 2q$ on the vertex set
$$V(K_{p/q})=\{0,1,\dots,p-1\}$$ with edges set
$$E(K_{p/q})=\{\{a,b\}: q\le |a-b|\le p-q\}.$$
\end{defi}

For the relevance of these graphs concerning the circular chromatic number, we refer the reader to Subchapter 6.1 of Hell and Ne\v{s}et\v{r}il's book \cite{HN}. One result proven in that section, which we will use, is the following.

\begin{thm}\label{thm:HNbol}{\rm (see as Theorem 6.3 in \cite{HN})}
The homomorphism $K_{p/q}\to K_{p'/q'}$ exists if and only if $p/q\le p'/q'.$
(Note that $p/q\ge 2$ and $p'/q'\ge 2$ are assumed.)
\end{thm}

\begin{thm}\label{thm:Kpq}
The graph $K_{p/q}$ admits an orientation without alternating odd cycles if and only if $p/q\le 3$.
\end{thm}

\proofka
The fact that $K_{p/q}$ admits an orientation without alternating odd cycles when $p/q\le 3$ follows from Proposition~\ref{prop:3chrom} and the well-known identity 
$\chi(K_{p/q})=\lceil p/q\rceil$, which also follows from Theorem~\ref{thm:HNbol}.

For the reverse statement, by Theorem~\ref{thm:HNbol} it suffices to consider the graph $K_{(3q+1)/q}$. 
Indeed, if $3q<p$, then $p\ge 3q+1$, implying that $p/q\ge (3q+1)/q$, and hence there exists a homomorphism $K_{(3q+1)/q}\to K_{p/q}$.
Therefore, if $K_{p/q}$ admitted a homomorphism to some $S_m$, the same would hold for $K_{(3q+1)/q}$.

Now we show indirectly that $K_{(3q+1)/q}$ does not admit the required orientation.
We imagine the vertices of $K_{(3q+1)/q}$ arranged in order around a circle, and refer to an edge $\{a,b\}$  as having length $j$ if the shorter circular distance between $a$ and $b$ is $j$. 
In the argument that follows, we consider only edges of lengths $q$ and $q+1$.
Observe that any two incident edges of length $q$ form a triangle with a third edge of length $q+1$, connecting their non-shared endpoints. 
To avoid alternating odd cycles, each such triangle must be oriented cyclically.
Moreover, since every edge of length $q$ appears in exactly two such triangles, these triangles form a larger cycle-like structure, in which every triangle shares an edge of length $q$ with the next one.
Therefore, all triangles must be consistently oriented; w.l.o.g., we assume they are all oriented clockwise. 
Now, consider the cycle formed by the sequence $q, 2q, q-1, 2q-1, q-2, 2q-2, ..., 0, q$, which alternates between moving $q$ steps clockwise and $q+1$ steps counterclockwise around the circle.
This cycle has length $2q+1$, and is an alternating odd cycle with only the vertex $q$ having outdegree $1$, contradicting our assumption.
\hfill$\Box$


\section{Making all shortest odd cycles alternating}\label{sect:altall}

In the previous section, we investigated whether a graph can be oriented so as to avoid alternating odd cycles entirely.
Here we consider a sort of complementary problem: whether a graph can be oriented so that at least all of its shortest odd cycles are alternating.
The odd girth (the length of the shortest odd cycle) of a graph $G$ will be denoted by $og(G)$.
When we have $og(G)=3$, then the answer is obviously affirmative: an acyclic orientation suffices.
This is because the $3$-cycle has only two non-isomorphic orientations: the cyclic and the alternating. 
Since this no longer holds for longer odd cycles, for those the question becomes more interesting.
First, we prove the following general statement, which will be useful when considering different graph classes.

\begin{lemma} \label{lem:bip_mat}
If the edge set of a graph G with $og(G)=5$ can be partitioned into a bipartite graph and a matching, then it admits an orientation in which all of its $5$-cycles are alternating.
\end{lemma}

Notice that every graph satisfying the conditions of Lemma \ref{lem:bip_mat} must be $4$-colorable as it is the union of two bipartite graphs.

\begin{proof}
    If we choose the bipartite subgraph to be maximal, then the remaining matching edges must connect vertices on the same side of the bipartite graph.
    Consequently, every odd cycle must contain an odd number of these matching edges. 
    In particular, a $C_5$ cannot contain three such edges, so it must contain exactly one.
    Now fix such a partition: a bipartite graph with vertex classes $A$ and $B$, together with a matching $M$ whose edges each lie entirely within $A$ or within $B$.
    Orient all bipartite edges from $A$ to $B$, and orient the edges of $M$ arbitrarily.
    Any $5$-cycle in $G$ contains exactly two vertices from $A$ or exactly two from $B$, and these two vertices act either both as sources or both as sinks in the directed cycle.
    Hence every $5$-cycle of $G$ contains at least two sources or at least two sinks, and therefore every $5$-cycle is alternating under this orientation.
\end{proof}

\begin{remark}
\label{rem:bip-mat}
    The statement of Lemma \ref{lem:bip_mat} can be extended to graphs with larger odd girth as follows. 
    If $G$ contains a bipartite subgraph such that every shortest odd cycle in $G$ uses at most one edge outside this subgraph, then $G$ admits an orientation in which every shortest odd cycle is alternating.
\end{remark}

\subsection{The Clebsch graph}

The Clebsch graph is well known for its role in establishing the lower bound for the Ramsey number $R(3,3,3)=17$.
Indeed, Greenwood and Gleason showed that $E(K_{16})$ can be $3$-colored without creating a monochromatic triangle, by partitioning the edges into three copies of the Clebsch graph.

As an application of Lemma \ref{lem:bip_mat}, we can show an orientation of the Clebsch graph containing only alternating $5$-cycles.

\begin{prop}
    The Clebsch graph has odd girth $5$, and it can be oriented so that all copies of $C_5$ become alternating.
\end{prop}

\begin{proof}
    By Greenwood and Gleason's \cite{GG} lower bound construction for $R(3,3,3)=17$ it follows that the Clebsch graph contains no triangle.
    It is easy to see on the embedding shown in Figure \ref{fig:Cl}, that it contains $C_5$
    Thus its odd girth is indeed $5$.
    
    For the second statement, by Lemma \ref{lem:bip_mat} it suffices to show that the edge set of the Clebsch graph can be partitioned into a bipartite graph on parts $A$ and $B$, and a matching.   
    Such a decomposition is illustrated in Figure~\ref{fig:Cl}: the dashed edges correspond to the matching, while the solid edges form the bipartite subgraph.
    Vertices of $A$ are drawn as dots, and those of $B$ as squares.
    Note that the figure also displays an orientation, constructed using Lemma~\ref{lem:bip_mat}, that guarantees every $5$-cycle in the Clebsch graph is alternating.
\end{proof}

\begin{figure} [h!]
    \centering
    \includegraphics[width = 12 cm]{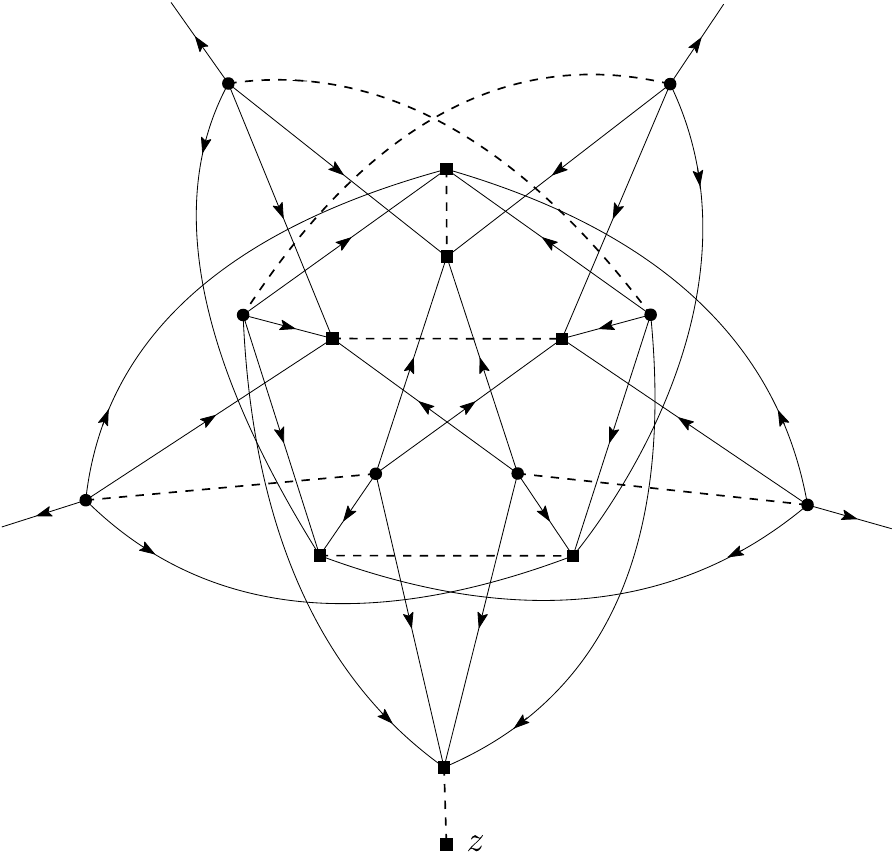}
    \caption{An orientation of the Clebsch graph with all $5$-cycles alternating. The dashed edges may be oriented arbitrarily. The external vertex $z$ is also connected to the other four edges that are drawn without endpoints.}
    \label{fig:Cl}
\end{figure}

\begin{cor}
    The Grötzsch graph, i.e., the Mycielskian $M_2(C_5)$, appears as a subgraph of the Clebsch graph.
    It is triangle-free, yet, as the Mycielskian of the $5$-cycle, it necessarily contains a $C_5$, so its odd girth is $5$.
    Consequently, the Grötzsch graph can also be oriented so that all of its shortest odd cycles are alternating.
\end{cor}

\subsection{Kneser graphs}

In this subsection we answer the question considered in this section for a family of Kneser graphs.
Note that the odd girth of Kneser graphs is known \cite{PT}, and in fact, it follows from a similar argument that the same value holds for Schrijver graphs as well:
\[
    og(\KG(n,k)) = og(\SG(n,k)) = 2\Big\lceil \frac{k}{n-2k} \Big\rceil + 1.
\]
The following result about Kneser graphs, of course, implies the analogous statement also for Schrijver graphs. 

\begin{thm}\label{thm:KGspec}
For any positive integers $m$ and $k$ the Kneser graph $\KG(m(2k+1),mk)$ admits an orientation in which all its shortest odd cycles become alternating.
\end{thm}

\proofka
Note that, using the above formula, we have 
$$\KG(m(2k+1),mk) = 2\Big\lceil\frac{mk}{m(2k+1)-2mk}\Big\rceil+1 = 2k+1. $$ 
Thus, we want an orientation where each copy of $C_{2k+1}$ is alternating.

Fix an arbitrary element $j\in [m(2k+1)]$ and orient the Kneser graph so that every vertex corresponding to an $mk$-element subset that contains $j$ becomes a source (i.e., all incident edges point away from it). 
Edges not incident to such vertices can be oriented arbitrarily. 
We prove that all cycles $C_{2k+1}\subseteq\KG(m(2k+1),mk)$ are alternating under this orientation. 
Note that this follows if we show that every $C_{2k+1}$ subgraph contains at least $k$ source vertices. 
Consider a fixed $(2k+1)$-cycle $C$ in the Kneser graph, and look at the set of pairs $(i,A)$, where $i\in A$ and $A\in V(C)\subseteq V(\KG(m(2k+1),mk))={[m(2k+1)]\choose mk}$. 
There are $2k+1$ such sets $A$, each corresponding to a vertex of the fixed cycle $C$ and containing $mk$ elements.
So the number of $(i,A)$ pairs in our set is $(2k+1)mk$. 
On the other hand, since vertices representing intersecting $mk$-subsets are never adjacent, and the independence number of $C$ is $k$, each element $i\in [m(2k+1)]$ can belong to at most $k$ of the $mk$-subsets representing vertices of the cycle.
Therefore, the total number of pairs $(i,A)$ can be at most $m(2k+1)k$, with equality occurring only if each element $i\in [m(2k+1)]$ appears exactly $k$ times. 
Since our earlier count shows that this maximum is indeed reached, the said condition must hold. 
Thus our fixed element $j$ must appear in exactly $k$ vertices of every $C_{2k+1}$ subgraph, ensuring there are $k$ vertices with indegree zero. 
Hence all such subgraphs are indeed alternating odd cycles.
\hfill$\Box$

\subsection{Schrijver graphs}

In this subsection, we show a family of Schrijver graphs that can be oriented with all shortest odd cycles alternating.

\begin{thm}
    \label{thm:sg-all}
    For all $t\leq 4$, every $t$-chromatic Schrijver graph admits an orientation in which all of its shortest odd cycles are alternating.
\end{thm}

\begin{proof}
The claim is immediate for $t \leq 2$.
The case $t=3$ follows directly from Theorem \ref{thm:KGspec}.
Indeed, when $m=1$, the theorem states that every $3$-chromatic Kneser graph admits an orientation in which all copies of its shortest odd cycle are alternating.
Since every $3$-chromatic Schrijver graph is an induced subgraph of the corresponding Kneser graph having the same odd girth, the same orientation works for the Schrijver graph as well.

It remains to prove the statement for $4$-chromatic Schrijver graphs, i.e., $\SG(2k+2,k)$ for any $k>0$.
We examine the cases of odd and even $k$ separately.
Note that we have $og(\SG(2k+2,k))=k+2$ for odd $k$ and $og(\SG(2k+2,k))=k+1$ for even $k$.
Our proof relies on the well-understood structure of 4-chromatic Schrijver graphs, as detailed in \cite{BB1,BB2,STActa}.

For odd $k$, the graph $\SG(2k+2,k)$ consists of
\begin{itemize}
    \item $\frac{k+1}{2}$ layers, each containing $C_{2k+2}$,
    \item edges connecting corresponding vertices of neighboring layers,
    \item edges connecting opposite vertices in the bottom (or first) layer, completing a Möbius ladder $M_{2k+2}$, and
    \item edges extending the top (or last) layer to a complete bipartite graph.
\end{itemize}
See for example $\SG(12,5)$ in Figure \ref{fig:Sg}.
The odd girth of $\SG(2k+2,k)$ is $k+2$, so our goal is to find an orientation of the Schrijver graph such that each $C_{k+2}$ becomes alternating.

\begin{figure} [h!]
    \centering
    \includegraphics[width = 6 cm]{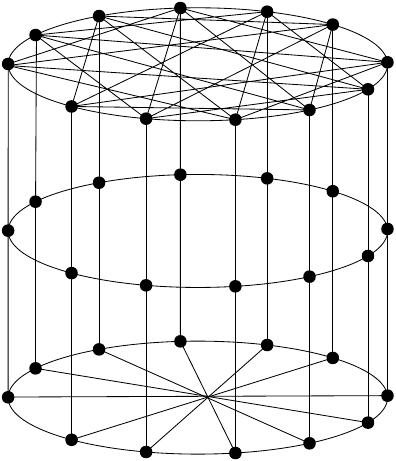}
    \caption{The Schrijver graph $\SG(12,5)$ consists of $3$ layers: a Möbius ladder $M_{12}$ in the bottom, a cycle $C_{12}$ in the middle and a complete bipartite graph $K_{6,6}$ at the top.}
    \label{fig:Sg}
\end{figure}

\medskip\par\noindent
\begin{cl*}
Any cycle of length $k+2$ uses exactly one diagonal edge from the bottom layer.
\end{cl*}
\begin{proof}
Note that any copy of a $C_{k+2}$ contains at least one diagonal edge from the Möbius ladder in the bottom layer.
Indeed, if we remove all $k+1$ diagonal edges, the bottom layer transforms into a cycle $C_{2k+2}$, and thus the remaining structure is bipartite.
This can be easily verified via a greedy $2$-coloring strategy, and thus we see that the remaining graph contains no odd cycles.

To prove the claim, we show that no cycle of length $k+2$ uses at least two diagonal edges from the bottom layer.
First observe that every cycle using at least two diagonal edges from the bottom layer and one from the top layer must have length greater than $k+2$.  
Indeed, such a cycle necessarily contains at least two diagonal edges from the bottom layer, some edge(s) connecting two endpoints of the diagonal edges; $k-1$ vertical edges between the top and bottom layers, and one edge from the top layer.  
This already accounts for $k+3$ edges, so the cycle cannot be of length $k+2$.

Therefore, we can project the cycle down onto the bottom layer, where it forms a closed sequence of edges (possibly using some edges multiple times) with length $k+2$.
Since the sequence is closed and has odd length, it must contain an odd cycle as a subgraph. 
But the odd girth of the Schrijver graph is $k+2$, so this closed sequence must itself be a cycle of length $k+2$. 

Now observe that deleting two opposite edges from the $C_{2k+2}$ part of $M_{2k+2}$ yields a bipartite graph.  
Consequently, every odd cycle in $M_{2k+2}$ must contain at least one edge from each opposite pair of edges of $C_{2k+2}$.  
Therefore, any cycle of length $k+2$ on the bottom layer must use at least $k+1$ edges of the base cycle $C_{2k+2}$, leaving room for at most one diagonal edge.  
Hence, the claim is proven.   
\end{proof}

Now we partition the edges of $\SG(2k+2,2)$ into a matching and a bipartite graph, so that any shortest odd cycle uses exactly one edge of the matching.
Specifically, we take the diagonal edges of $M_{2k+2}$, forming a matching, and the remaining edges form a bipartite graph.
By Remark \ref{rem:bip-mat}, such a partitioning proves the existence of an orientation, where each shortest odd cycle is alternating.

Thus, the statement is proven for odd values of $k$.
We note that the even case could be proved similarly by examining the structure of the Schrijver graphs and their shortest odd cycles.
However, we omit this detailed argument here, as the result also follows from a special case of Theorem \ref{thm:KGspec}.
Specifically, for all $k$, the Schrijver graph $\SG(2k+2,k)$ is a subgraph of the Kneser graph $\KG(2k+2,k)$ with the same odd girth.
Therefore, by setting $m=2$, the statement for even $k$ follows directly from Theorem \ref{thm:KGspec}. 
\end{proof}

\begin{ques}
    Which $t$-chromatic Schrijver graphs with $t\geq 5$ admit an orientation so that all of their shortest odd cycles become alternating?
\end{ques}






\section{Acknowledgement}
The authors thank Gábor Tardos for his invaluable remarks and for his generous support in completing this paper.

\newpage
\section{Questions}

Several questions remain open in this area. Here, we highlight a selection of those that we find most intriguing.

\begin{enumerate}
\item  Does there exist an integer $m$ such that $\KG(8,3) \rightarrow S_m$? More generally, for which integers $k \geq 3$ does there exist an $m$ such that $\KG(2k+2, k) \rightarrow S_m$?

\item Is it true for arbitrary $n$ and $k$, that there exists an orientation of $\KG(n,k)$ in which every shortest odd cycle is alternating? (The proof of Theorem~\ref{thm:KGspec} makes use of the ratio $n/k=(2x+1)/x$.)

\item Which $t$-chromatic Schrijver graphs with $t\geq 5$ admit an orientation so that all of their shortest odd cycles become alternating?

\item Is  $M_r(C_{2k+1})$ orientable in such a way that every shortest odd cycle is alternating? (For $M(C_5)$, we already know the answer is yes, since this appears in the orientation of the Clebsch graph.)

\item The orientability that avoids alternating odd cycles has been characterized by the existence of a homomorphism into some symmetric shift graph $S_m$. What about making all the shortest odd cycles alternating? Is there some nice characterization for this question as well?
\end{enumerate}


\bibliography{references}

\begin{thebibliography}{10}

\bibitem{BS}
S.~Baum and M.~Stiebitz.
\newblock Coloring of graphs without short odd paths between vertices of the
  same color class.
\newblock Unpublished manuscript, 2005.

\bibitem{BB1}
B.~Braun.
\newblock Symmetries of the stable {K}neser graphs.
\newblock {\em Advances in Applied Mathematics}, 45:12--14, 2010.

\bibitem{BB2}
B.~Braun.
\newblock Independence complexes of stable {K}neser graphs.
\newblock {\em Electronic Journal of Combinatorics}, 18:Paper 118, 2011.

\bibitem{CLZ}
G.~J. Chang, D.~F. Liu, and X.~Zhu.
\newblock A short proof for {C}hen's alternative {K}neser coloring lemma.
\newblock {\em Journal of Combinatorial Theory, Series A}, 120:159--163, 2013.

\bibitem{PAC}
P-A. Chen.
\newblock A new coloring theorem of {K}neser graphs.
\newblock {\em Journal of Combinatorial Theory, Series A}, 118:1062--1071,
  2011.

\bibitem{CPR}
B.~Codenotti, P.~Pudlák, and G.~Resta.
\newblock Some structural properties of low-rank matrices related to
  computational complexity.
\newblock {\em Theoretical Computer Science}, 235(1):89--107, 2000.

\bibitem{DM}
H.~R. Daneshpajouh and F.~Meunier.
\newblock Box complexes: At the crossroad of graph theory and topology.
\newblock {\em Discrete Mathematics}, 348:Article 114422, 2025.

\bibitem{EFHKRS}
P.~Erd\H{o}s, Z.~F\"uredi, A.~Hajnal, P.~Komj\'ath, V.~R\"odl, and \'A. Seress.
\newblock Coloring graphs with locally few colors.
\newblock {\em Discrete Mathematics}, 59:21--34, 1986.

\bibitem{EH}
P.~Erd\H{o}s and A.~Hajnal.
\newblock Some remarks on set theory. ix. combinatorial problems in measure
  theory and set theory.
\newblock {\em Michigan Mathematical Journal}, 11:107--127, 1964.

\bibitem{KF}
K.~Fan.
\newblock Evenly distributed subsets of $s^n$ and a combinatorial application.
\newblock {\em Pacific Journal of Mathematics}, 98(2):323--325, 1982.

\bibitem{GH}
A.~Golovnev and I.~Haviv.
\newblock The (generalized) orthogonality dimension of (generalized) {K}neser
  graphs: Bounds and applications.
\newblock {\em Theory of Computing}, 18:Article 22, 1--22, 2022.

\bibitem{GG}
R.~E. Greenwood and A.~M. Gleason.
\newblock Combinatorial relations and chromatic graphs.
\newblock {\em Canadian Journal of Mathematics}, 7:1--7, 1955.

\bibitem{GyJS}
A.~Gy\'arf\'as, T.~Jensen, and M.~Stiebitz.
\newblock On graphs with strongly independent colour-classes.
\newblock {\em Journal of Graph Theory}, 46:1--14, 2004.

\bibitem{HN}
P.~Hell and J.~Ne\v{s}et\v{r}il.
\newblock {\em Graphs and Homomorphisms}.
\newblock Oxford University Press, New York, 2004.

\bibitem{Kne}
M.~Kneser.
\newblock Aufgabe 300.
\newblock {\em Jahresbericht der {D}eutschen {M}athematiker-{V}ereinigung},
  58:27, 1955.

\bibitem{KPS}
J.~K\"orner, C.~Pilotto, and G.~Simonyi.
\newblock Local chromatic number and {S}perner capacity.
\newblock {\em Journal of Combinatorial Theory, Series B}, 95:101--117, 2005.

\bibitem{LZ}
D.~F. Liu and X.~Zhu.
\newblock A combinatorial proof for the circular chromatic number of {K}neser
  graphs.
\newblock {\em Journal of Combinatorial Optimization}, 32:765--774, 2016.

\bibitem{LLKn}
L.~Lov\'asz.
\newblock {K}neser's conjecture, chromatic number, and homotopy.
\newblock {\em Journal of Combinatorial Theory, Series A}, 25(3):319--324,
  1978.

\bibitem{Matbook}
J.~Matou\v{s}ek.
\newblock {\em Using the {B}orsuk-{U}lam Theorem}.
\newblock Springer-Verlag, Berlin, 2003.

\bibitem{MZ}
J.~Matou\v{s}ek and G.~M. Ziegler.
\newblock Topological lower bounds for the chromatic number: A hierarchy.
\newblock {\em Jahresbericht der Deutschen Mathematiker-Vereinigung},
  106(2):71--90, 2004.

\bibitem{PT}
S.~Poljak and Zs. Tuza.
\newblock Maximum bipartite subgraphs of {K}neser graphs.
\newblock {\em Graphs and Combinatorics}, 3(1):191--199, 1987.

\bibitem{Schr}
A.~Schrijver.
\newblock Vertex-critical subgraphs of {K}neser graphs.
\newblock {\em Nieuw Archief voor Wiskunde (3)}, 26(3):454--461, 1978.

\bibitem{SST}
T.~Schweser, M.~Stiebitz, and B.~Toft.
\newblock {\em Brooks' Theorem. Graph Coloring and Critical Graphs}.
\newblock Springer Monographs in Mathematics. Springer, Cham, 2024.

\bibitem{ST1}
G.~Simonyi and G.~Tardos.
\newblock Local chromatic number, {K}y {F}an's theorem, and circular colorings.
\newblock {\em Combinatorica}, 26:587--626, 2006.

\bibitem{STcolful}
G.~Simonyi and G.~Tardos.
\newblock Colorful subgraphs in {K}neser-like graphs.
\newblock {\em European Journal of Combinatorics}, 28:2188--2200, 2007.

\bibitem{STdir}
G.~Simonyi and G.~Tardos.
\newblock On directed local chromatic number, shift graphs, and {B}orsuk-like
  graphs.
\newblock {\em Journal of Graph Theory}, 66:65--82, 2011.

\bibitem{STActa}
G.~Simonyi and G.~Tardos.
\newblock On 4-chromatic {S}chrijver graphs: their structure,
  non-3-colorability, and critical edges.
\newblock {\em Acta Mathematica Hungarica}, 161:583--617, 2020.

\bibitem{Stieb}
M.~Stiebitz.
\newblock {\em Beitr\"age zur {T}heorie der f\"arbungskritischen {G}raphen}.
\newblock Habilitation thesis, Technische Hochschule Ilmenau, 1985.

\end{thebibliography}
\bibliographystyle{plain}

\end{document}